\documentclass[12pt]{amsart}

\usepackage{amsmath,amsthm,amsfonts,latexsym,amssymb,amscd,enumerate,times}
\usepackage{fullpage}
\usepackage[all]{xy}

\newtheorem{theorem}{Theorem}[section]
\newtheorem{thm}[theorem]{Theorem}
\newtheorem{prop}[theorem]{Proposition}

\newtheorem{cor}[theorem]{Corollary}
\newtheorem{conj}[theorem]{Conjecture}

\newtheorem{ex}[theorem]{Example}

\newtheorem{defn}[theorem]{Definition}

\newtheorem{quest}[theorem]{Question}

\theoremstyle{definition}

\theoremstyle{remark}

\DeclareMathOperator{\ind}{ind}

\DeclareMathOperator{\scal}{scal}

\DeclareMathOperator{\ch}{{\rm ch}}

\DeclareMathOperator{\U}{U}

\DeclareMathOperator{\Mat}{Mat}

\DeclareMathOperator{\vol}{vol}

\DeclareMathOperator{\Mor}{Mor}

\newcommand{\C}{\mathbb{ C}}

\newcommand{\coker}{{\rm coker}\,} 

\newcommand{\N}{\mathbb{ N}}

\newcommand{\Z}{\mathbb{ Z}}

\newcommand{\Hom}{\operatorname{Hom}}

\newcommand{\KK}{K\!K}

\newcommand{\Q}{\mathbb{ Q}}
\newcommand{\R}{\mathbb{ R}}

\newcommand{\rk}{\operatorname{rk}}

\newcommand{\ra}{{\rightarrow}}

\newcommand{\id}{\operatorname{id}}

{\catcode`@=11\global\let\c@equation=\c@theorem}

\date{\today}

\begin{document}

\title{Positive scalar curvature, K-area  and essentialness}

\author{Bernhard Hanke}
\address{Institut f\"ur Mathematik, Universit\"at Augsburg, 86135 Augsburg, Germany}
\email{hanke@math.uni-augsburg.de}

\begin{abstract} The Lichnerowicz formula yields an index theoretic obstruction to positive scalar 
curvature metrics on closed spin manifolds. The most general form of this obstruction is due to 
Rosenberg and takes values in the $K$-theory of the group $C^*$-algebra of the fundamental group of the underlying manifold. 

We give an overview of recent results  clarifying the relation of the Rosenberg index to notions from large 
scale geometry like enlargeability and essentialness.  One central topic is the concept of 
$K$-homology classes of infinite $K$-area. This notion, which in its original 
form is  due to Gromov,  is put in a general context and 
systematically used as a link between geometrically defined large scale 
properties and index theoretic considerations. In particular, we prove essentialness and the non-vanishing of the 
Rosenberg index for manifolds of infinite $K$-area.

\end{abstract}

\maketitle


\section{Introduction and summary}

One of the fundamental problems in Riemannian geometry is to investigate the types of Riemannian metrics that 
exist on a given closed smooth manifold. It turns out that the signs of the associated 
curvature invariants distinguish classes of Riemannian manifolds with considerably
different geometric and topological properties. Usually the class of manifolds admitting metrics with  negative curvature is 
``big'' and the one with positive curvature is ``small''. The general existence theorems for negative 
Ricci curvature metrics \cite{Lohkamp} and negative scalar curvature metrics \cite{Yamabe}, 
the classical theorem of Bonnet-Myers on the finiteness of the fundamental groups of closed Riemannian
manifolds with positive Ricci curvature, Gromov's 
Betti number theorem for closed manifolds of non-negative sectional curvature \cite{Gromov_Curv_Diam_Betti}, 
the recent classification of manifolds with positive curvature  operators \cite{BW} and the proof of the 
differentiable sphere theorem \cite{BS(2008), BS(2009)} are prominent illustrations of this empirical fact.

In this context one may formulate two goals. The first is to 
develop methods to construct Riemannian metrics with 
distinguished properties on general classes of smooth manifolds. Important examples are the powerful tools in the theory of 
geometric partial differential equations, the surgery method due to  Gromov-Lawson \cite{GL(1980)}  and Schoen-Yau \cite{SY(1979)} for the construction of positive scalar curvature metrics, and methods based on geometric flow equations. The second 
deals with the formulation of  (computable) obstructions to the existence of Riemannian metrics with 
specific properties. Often this happens in connection with topological invariants associated 
to the given manifold like homology and homotopy groups and related data. These two goals are usually 
not completely seperate from each other in that they can result in overlapping questions, concepts and methods.
For example the Ricci flow is used to 
produce metrics with special properties, which a posteriori
determine the topological type of the underlying manifold.

Here we shall concentrate on the most elementary curvature invariant associated to a 
Riemannian manifold $(M,g)$, the scalar curvature $\scal_g: M \to \R$. This is usually defined 
by a twofold contraction of the Riemannian curvature tensor of   $(M,g)$, but also has a geometric 
interpretation in terms of the deviation of the volume growth of geodesic balls in $M$ 
compared to  geodesic balls in Euclidean space: 
\[
  \frac{ \vol_{(M^n,g)}(B_p(\epsilon))}{\vol_{(\R^n,g_{eucl})}(B_0(\epsilon))} = 1 -  
   \frac{\scal_g(p)}{6(n+2)}\cdot \epsilon^2  + O(\epsilon^4) \, . 
\]
Given a closed smooth manifold $M$ we shall study whether $M$ admits 
a Riemannian metric $g$ of positive scalar curvature, i.e.~satisfying $\scal_g(p) > 0$ for all $p \in M$. 
In view of the preceding description and the previous remarks it is on the one hand plausible that the resulting 
``inside bending of $M$''  at every point  
might put topological restrictions on $M$. On the other hand the scalar curvature involves an 
averaging process over sectional curvatures of $M$ so that a certain variability 
of the precise geometric shape and the topological properties of $M$ can be expected. 

In connection with  the positive scalar curvature question both aspects, the obstructive and constructive side, 
play  important roles and have
led to a complex body of mathematical insight with connections to index theory, geometric 
analysis, non-commutative geometry, surgery theory, bordism theory and stable homotopy theory.
The paper \cite{Rosenberg(2007)} gives  a comprehensive survey of the subject.  As such 
it represents not only an interesting geometric field of its own, but serves as a unifying link between 
several well established areas in geometry, topology and analysis. 

For metrics of positive scalar curvature there are two important obstructions, whose relation to each other 
is still not completely understood. One is based on the method of minimal hypersurfaces \cite{SY(1979)}
and the other on the analysis 
of the Dirac operator and index theory \cite{Lich(1963)}. 

In some sense the former obstruction is more elementary than the latter as it can be shown by a direct
calculation \cite{SY(1979)} that a nonsingular minimal hypersurface in a positive scalar curvature manifold admits itself 
a metric of positive scalar curvature. In connection with results from 
geometric measure theory that provide nonsingular minimal hypersurfaces representing 
codimension one homology classes in manifolds of dimension at most eight \cite{Smale(1993)}, 
this can inductively be used to exclude the existence of positive scalar curvature metrics 
on tori up to dimension eight, for instance. In higher dimensions the discussion of 
singularities on minimal hypersurfaces representing 
codimension one homology classes is a subtle topic and the subject of recent work 
of Lohkamp \cite{CL(2006), Lo1, Lo3}. This theme, which has important connections to the positive mass theorem in general 
relativity, will not be pursued further in our paper.

The second, index  theoretic, obstruction is both more restrictive as it requires a spin structure on the underlying manifold 
(or at least its universal cover), and less elementary as it is based on the Atiyah-Singer index theorem.  
In its most basic form it says that closed spin manifolds with non-vanishing $\hat{A}$-genus do 
not admit metrics of positive scalar curvature, the $\hat{A}$-genus being an integer (in the spin case) which depends on 
the rational Pontrjagin classes of the underlying manifold and its orientation class and hence only
on its oriented homeomorphism type. 

This obstruction was refined by Hitchin \cite{Hitchin(1974)} and  
Rosenberg \cite{Rosenberg(1983)} and in its most general form takes values 
in $KO_*(C^*_{\R, max} \pi_1(M))$, the $K$-theory of the real maximal group $C^*$-algebra of the fundamental group of the underlying manifold. 
It therefore touches important questions in noncommutative geometry linked to the Baum-Connes
and Novikov conjectures. The Gromov-Lawson-Rosenberg 
conjecture predicts that  for closed spin manifolds of dimension at least five the vanishing of this index obstruction is 
not only necessary, but also sufficient for 
the existence of a positive scalar curvature metric. Despite the fact that
this conjecture is wrong in general \cite{Schick(1998)}, 
the index obstruction being surpassed by the minimal hypersurface obstruction in some cases, it is 
remarkable that it holds 
for simply connected manifolds \cite{Stolz(1992)} and - in a stable 
sense - for all spin manifolds for which the assembly map 
with values in the $K$-theory of the real group $C^*$-algebra of the fundamental group is 
injective \cite{Stolz(2001)}, see Theorem \ref{Stolz} below. 
It is up to date unknown whether this conjecture in its original, unstable, form is true for spin manifolds 
with finite fundamental groups, although in this case the injectivity of the assembly map is known.
The index theoretic obstruction to positive scalar 
curvature will be recalled in Section \ref{index} of our paper. 

Gromov and Lawson used the index of the usual Dirac operator on closed spin manifolds twisted with bundles of small curvature to prove that some manifolds with  
vanishing $\hat{A}$-genus do still not admit positive scalar curvature metrics. For this aim they introduced
several kinds of largeness properties  for Riemannian manifolds, the most important ones being the notion of 
enlargeability \cite{GL(1980b), GL(1983)} and infinite $K$-area \cite{Gromov(1995)}. These properties have an asymptotic character in that they require, for each $\epsilon > 0$, the existence 
of a certain geometric structure attached to the underlying manifold which is $\epsilon$-small 
in an appropriate sense.  Precise definitions will be given in Section \ref{index} below. 

In light of the common index theoretic origin of these obstructions 
it is reasonable to expect that they are related to the Rosenberg index. In the papers 
\cite{HKRS(2007),HS(2006),HS(2007)} it is proved that the Rosenberg obstruction indeed subsumes the 
enlargeability obstruction in the sense that the former is non-zero for enlargeable spin manifolds. Moreover, it 
was shown in the cited papers that enlargeable manifolds are {\em essential}, i.e.~the classifying maps  of 
their universal covers map the homological fundamental classes to non-zero classes in the homology of the 
fundamental groups. This notion was introduced by Gromov in  \cite{Gromov(1983)} in connection with 
the systolic inequality giving an upper bound of the length of the shortest noncontractible 
loop in a Riemannian manifold $M$ in terms of the volume of $M$. In particular it follows from these 
results that enlargeable manifolds 
obey Gromov's systolic inequality. 

The methods introduced in \cite{HS(2006), HS(2007)} were applied in \cite{HS(2008)} 
to prove some cases of the strong Novikov conjecture. This is implied by the Baum-Connes conjecture  
and predicts that for discrete groups $G$ the rational assembly map 
\[
    K_*(BG) \otimes \Q \to K_*(C^*_{max} G) \otimes \Q 
\]
is injective. In loc.~cit.~it is shown that this map is indeed non-zero on all classes in $K_*(BG)\otimes \Q$ which
 are detected by classes in the subring generated by $H^{\leq 2}(BG; \Q)$. 
As a corollary higher signatures associated to elements in this 
subring of  $H^*(BG; \Q)$   are oriented homotopy invariants, a fact which had been proven first by 
Mathai \cite{Mathai}.

It turns out that the methods of \cite{HS(2006), HS(2008)} fit very nicely the concept of $K$-area introduced by 
Gromov in \cite{Gromov(1995)}. It is one purpose of the paper at hand to elaborate on this connection.  
Our main result, Theorem \ref{Kess}, states that $K$-homology classes of {\em infinite $K$-area} 
in closed manifolds $M$ map nontrivially to  $K_*(C^*_{max} \pi_1(M))$ 
under the assembly map. Generalizing the original concept of Gromov we call a $K$-homology class  
of {\em infinite $K$-area}, 
if it can be detected by bundles of finitely generated Hilbert $A$-modules equipped with holonomy representations 
which are arbitrarily close to the identity, where $A$ is some $C^*$-algebra with unit. Precise 
definitions are given in Section \ref{Karea} below, see in particular Definition \ref{Kflach}.  

From Theorem \ref{Kess} the main results of the papers \cite{HKRS(2007), HS(2006), HS(2007), HS(2008)}
follow quite directly. 
Apart from this we will demonstrate that closed spin manifolds whose $K$-theoretic fundamental classes are
 of infinite $K$-area have non-vanishing 
Rosenberg index (Corollary \ref{cor1}) and oriented manifolds with fundamental classes of infinite $K$-area are essential 
(Theorem \ref{thm2}). The first result solves a problem stated in the introduction of \cite{Listing(2010)}. 

In \cite{BrunnHan} essentialness is discussed from  a purely homological point of view. Among other things 
it is proved that the property of being enlargeable depends only on the image of the 
homological fundamental class of the underlying manifold 
in the rational homology of its fundamental group. This flexible formulation allows the construction of 
manifolds which are essential, but not enlargeable. We will briefly review these results in Section \ref{large_homology}. 
We do not know whether a proof of Theorem \ref{thm2} is feasible which avoids the ``infinite product construction'' laid 
out in  Section \ref{Karea}. Also, we do not know an essential manifold whose
fundamental class is not of infinite $K$-area, see Question \ref{final_problem}. 

This paper is intended on the one hand as a report on recent results pertaining to the positive scalar curvature question in the light of methods from index theory, $K$-theory and asymptotic geometry as obtained by the author and his coauthors.  
On the other hand it is meant to establish the point of view that the 
notion of infinite $K$-area may serve as a unifying principle for these results, which sometimes 
allows short and conceptual proofs.  

I am grateful to the DFG Schwerpunkt ``Globale Differentialgeometrie'' for financial 
support during the last years. Special thanks go to Thomas Schick for a very fruitful and pleasant collaboration.
Most of the material in these notes is based on ideas developped during this collaboration. 

Daniel Pape carefully read the first version of this manuscript and helped to improve the presentation with 
many useful comments.

\section{Index obstruction to positive scalar curvature}  \label{index}

The Gau\ss -Bonnet formula implies that closed surfaces with nonpositive Euler characteristic 
do not admit positive scalar curvature metrics. These comprise all closed surfaces apart from the 
two sphere and the real projective plane. The mechanism 
behind this obstruction is the fact that a topological invariant, the Euler characteristic, may be expressed 
as an integral over a curvature quantity, the Gau\ss~curvature. 

In higher dimensions obstructions to positive scalar curvature metrics can be obtained in a more indirect 
way by use of the Atiyah-Singer index theorem. Let $M$ be a closed smooth oriented manifold of dimension divisible 
by four. The $\hat{A}$-genus $\hat{A}(M)\in \Q$ of $M$ is obtained by evaluating the $\hat{A}$-polynomial
\[
   \mathcal{\hat{A}}(M) =  1 - \frac{p_1(M)}{24} + \frac{ - 4 p_2(M)  + 7p_1^2(M) }{2^7 \cdot 3^2 \cdot 5} 
   + \ldots 
\]
in the Pontrjagin classes of $M$ on  the fundamental class of $M$. This is an invariant of the oriented 
homeomorphism type of $M$ 
by the topological invariance of rational Pontrjagin classes. It is 
an integer,  if $M$ is equipped with a spin structure. This is implied by  the fact that in 
this case the Atiyah-Singer index theorem provides an equation
\[
   \hat{A}(M) = \ind(D_g^+) = \dim_{\C} (\ker D_g^+) - \dim_{\C} ( \coker D_g^+ )
\]
where 
\[
   D_g^{\pm}  : \Gamma(S^{\pm} ) \to \Gamma(S^{\mp} ) 
\]
is the Dirac operator on the complex spinor bundle $S = S^+ \oplus S^- \to M$ of $(M,g)$. Here $g$ is 
an arbitrary Riemannian metric on $M$. Due to the appearance of $g$ 
in the definition of $D_g^+$, the Atiyah-Singer index theorem
relates topological to geometric properties  of $M$.
Detailed information on the definition of $D_g^+$ and 
spin geometry in general can be found in \cite{LawsonMichelsohn}. 

The Bochner-Lichnerowicz-Weitzenb\"ock formula \cite{Lich(1963)}  
\[
    D_g^- \circ D_g^+ = \nabla^* \nabla + \frac{\scal_g}{4}
\]
implies that if $\scal_g (M) > 0$, then the Dirac operator $D^+_g$ is invertible and hence $\ind(D_g^+) = 0$. 
From this we obtain the following fundamental result, see \cite[Theor\`eme 2]{Lich(1963)}.  

\begin{thm} Let $M$ be a closed spin manifold with 
$\hat{A}(M) \neq 0$. Then $M$ does not admit a metric of positive scalar curvature. 
\end{thm}

However, the vanishing of this obstruction is not sufficient for the existence of positive scalar curvature metrics. For example, 
the $\hat{A}$-genus of the $4k$-dimensional torus 
$T^{4k}$ vanishes for all $k > 0$, because these manifolds are parallelizable.

The index theoretic approach explained above can be refined by 
considering the twisted Dirac operator 
\[
   D_{g,E}^+ : \Gamma(S^+ \otimes E) \to \Gamma(S^- \otimes E) 
\]
where $E \to M$ is some finite dimensional Hermitian vector bundle equipped with a Hermitian 
connection, cf. \cite[Prop. II.5.10]{LawsonMichelsohn}. 
The Atiyah-Singer index theorem computes the index of this operator as
\[
   \ind(D_{g,E}^+) = \langle \mathcal{\hat{A}}(M)  \cup \ch(E) , [M] \rangle \in \Z \, . 
\]
Due to the appearance of the Chern character $\ch(E) \in H^{ev}(M;\Q)$
 this number can be non-zero even though $\hat{A}(M)$ vanishes. Unfortunately, the nonvanishing of $\ind(D_{g,E}^+)$ does not obstruct positive 
scalar curvature metrics on $M$ as the following example shows. 

\begin{ex} \label{badex} Let $M^n = S^{4k+2}$. Because the Chern character defines 
an isomorphism 
\[
  \ch :  K^0(M) \otimes \Q \cong H^{ev}(M;\Q) \, , 
\]
there is a finite dimensional Hermitian bundle $E \to M$ with $\ch_{2k+1}(E) \neq 0 \in H^n(M;\Q)$. Hence, for 
any connection on $E$ and any choice of Riemannian metric $g$ on $M$, we get 
$\ind(D_{g,E}^+) \neq 0$ although $M$ admits a metric of positive scalar curvature.  
\end{ex}  

This is due to the fact that now the Bochner-Lichnerowicz-Weitzenb\"ock formula 
\[
    D_{g,E}^- \circ D_{g,E}^+ = \nabla^* \nabla + \frac{\scal_g}{4} + R^E 
\]
contains an additional operator $R^E : \Gamma(S^{\pm} \otimes E) \to \Gamma(S^{\pm} \otimes E)$ 
of order $0$ which depends on the curvature 
of the bundle $E$, cf. \cite[Theorem 8.17]{LawsonMichelsohn}, so that 
even in the case when $\scal_g > 0$, the operator $D^+_{g,E}$ may not be invertible. 

Gromov and Lawson observed in \cite{GL(1980b)} that this method 
does still lead to an effective obstruction to positive scalar curvature metrics on $M$ in case that for all $\epsilon$ 
there is a twisting bundle $E \to M$ which satisfies $\| R^E \| < \epsilon$ and 
whose Chern character contributes nontrivially to $\ind(D^+_{g,E})$. 
If in this case $M$ carried a metric $g$ satisfying $\scal_g > 0$ we would find a twisting bundle $E$ with 
\[
   \| R^E\| < \frac{\min_{p \in M} | \scal_g(p)|}{4}
\]
and the Bochner-Lichnerowicz-Weitzenb\"ock  formula would then imply that $\ind D^+_{g,E}=0$, a contradiction. 

For example this reasoning can be used to show 
that the tori $T^n$ do not admit metrics of positive scalar curvature \cite{GL(1980b)}.

A general class of manifolds where twisting bundles with the described property can be found 
are {\em enlargeable} manifolds, which were introduced in loc.~cit., and 
manifolds of infinite $K$-area in the sense of \cite{Gromov(1995)}. 
We will discuss these notions and put them in a general context in Section \ref{Karea}. 

The index theoretic point of view was refined by Rosenberg \cite{Rosenberg(1983), Rosenberg(1986)}. 
For any discrete group $G$ the group $C^*$-algebra $C^*G$ is constructed by completing 
 the group algebra $\C[G]$ 
with respect to some pre-$C^*$-norm coming 
from unitary representations of $G$ on a Hilbert space and taking the induced embedding of $\C[G]$ 
into the bounded operators on this Hilbert space. More specifically, if one starts with the regular representation of $G$ on the 
space of square summable functions $l^2(G)$ this leads to the {\em reduced group $C^*$-algebra} $C^*_{red} G$ 
and taking all unitary representations of $G$ into account one arrives at the {\em maximal group $C^*$-algebra} $C^*_{max} G$.
For more details we refer to \cite{Black, HR, Wegge-Olsen}. 
These $C^*$-algebras and their $K$-theories are in general different \cite[Exercise 12.7.7]{HR}, but the following 
construction 
works for both variants, and this is why we drop the subscript  from our notation.  Note that the left 
translation action of $G$ on 
$\C[G]$ induces a  left $G$-action on $C^* G$. 

Let $M$ be a closed spin manifold of even dimension. 
The Mishchenko-Fomenko bundle $E \to M$ is defined as 
\[
   E = \widetilde{M} \times_{\pi_1(M)} C^* \pi_1(M) \, . 
\]
It is a locally trivial bundle of free right Hilbert $C^* \pi_1(M)$-modules of rank one in the sense 
of \cite{Schick(2005), Wegge-Olsen}.  The fibrewise inner product is induced by the canonical 
inner product 
\begin{eqnarray*} 
    C^* \pi_1(M) \times C^* \pi_1(M) & \to & C^* \pi_1(M) \\
         (x,y)                 &    \mapsto & x^* \cdot y \, . 
\end{eqnarray*}
By construction the bundle $E \to $M can be equipped with a flat connection. Depending on the choice of a metric $g$  on $M$ we obtain a twisted  
Dirac operator 
\[
   D^+_{g,E} : \Gamma(S^+ \otimes E ) \to \Gamma(S^- \otimes E) 
\]
with an index 
\[
 \alpha(M) := \ind(D^+_{g,E}) =  \ker (D^+_{g,E})  - \coker (D^+_{g,E}) \in K_0(C^* \pi_1(M)) \, .
\]
The group $K_0(C^* \pi_1(M))$ consists of formal differences of finitely generated projective 
$C^* \pi_1(M)$-modules, cf.~\cite{Black}. For the infinite dimensional 
twisting bundle $E$ the modules $\ker(D^+_{g,E})$ and $\coker(D^+_{g,E})$ are not in this class 
in general, but this holds  after a $C^* \pi_1(M)$-compact perturbation
of $D^+_{g,E}$ which makes this operator a $C^* \pi_1(M)$-Fredholm operator. 
For precise formulations and 
more details on the involved theory we refer the reader to \cite{MF(1979)}, 
in particular to  Theorem 3.4.~in loc.~cit.

It follows again from the Bochner-Lichnerowicz-Weitzenb\"ock  formula (which does not contain 
a curvature term $R^E$ as $E$ is flat) that the index 
$\alpha(M) \in K_0(C^* \pi_1(M)) $ vanishes, if $\scal_g > 0$. Moreover, the Mishchenko-Fomenko 
index theorem \cite{MF(1979)} implies that - similar to the invariant $\hat{A}(M)$ - 
the obstruction $\alpha(M)$ does  not depend on the choice of a Riemannian metric on $M$, 
but only on the oriented homeomorphism type of $M$.

There is an alternative construction of $\alpha(M)$ based on analytic $K$-homology \cite{Black, HR}.
As before let $M$ be a closed spin manifold. We do no longer assume that $n := \dim M$ is even
(this only simplified the above considerations). 

In this setting $\alpha(M)$ is defined as the image of the $K$-theoretic fundamental class $[M]_K \in K_n(M)$ which 
is induced by the given spin structure under the composition 
\[
    K_n(M) = K_n^{\pi_1(M)}( \widetilde{M}) 
    \to K_n^{\pi_1(M)}( \underline{E} \pi_1(M)) \stackrel{\mu}{\ra} K_n(C^* \pi_1(M)) \, . 
\]
Here the first map is induced by the $\pi_1(M)$-equivariant 
classifying map $\widetilde{M} \to \underline{E} \pi_1(M)$ from the universal cover of $M$ to the universal 
contractible $\pi_1(M)$-space with finite isotropy groups and the second map is the Baum-Connes assembly map, 
cf. \cite{Black}.  

There is a real analogue $\alpha_{\R}(M)$ of the index obstruction $\alpha(M)$ which, 
for simply connected manifolds, was introduced
in the paper \cite{Hitchin(1974)} and is defined as the image of the $KO$-theoretic fundamental class 
$[M]_{KO} \in KO_n(M)$ under the composition
\[
    KO_n(M) = KO^{\pi_1(M)}_n( \widetilde{M}) 
    \to KO^{\pi_1(M)}_n( \underline{E} \pi_1(M)) \stackrel{\mu}{\ra} KO_n(C^* \pi_1(M)) \, . 
\]
The invariant $\alpha_{\R}(M)$ is more sensitive to differential topological 
properties of $M$ than $\alpha(M)$. 
For example it is different from zero on some exotic spheres \cite{Hitchin(1974)}. A 
refined variant of the Bochner-Lichnerowicz-Weitzenb\"ock argument shows that 
$\alpha_{\R}(M) = 0$, if $M$ admits a metric of positive scalar curvature.

In case we are dealing with the reduced group $C^*$-algebra $C^*_{red} \pi_1(M)$, the vanishing of the $\alpha$-obstruction 
is closely linked to properties of the Baum-Connes  assembly map 
\[
   \mu_{\R} :  KO^G_*( \underline{E}  G) \to KO_*(C^*_{red} G) 
\]
and its complex analogue
\[
   \mu_{\C}  : K^G_*( \underline{E}  G) \to K_*(C^*_{red} G) \, . 
\]
According to the Baum-Connes conjecture \cite{Black}, a central open problem in noncommutative geometry, these 
two maps are isomorphisms for all discrete groups $G$. 

The following conjecture has played a prominent role in the subject. It expresses the 
expectation that the Rosenberg obstruction is in some sense optimal.

\begin{conj}[Gromov-Lawson-Rosenberg conjecture] Let $M$ be a closed 
spin manifold of dimension at least five and with $\alpha_{\R}(M) = 0$. Then $M$ admits a 
metric of positive scalar curvature.
\end{conj}

This is true, if $M$ is simply connected \cite{Stolz(1992)}, but wrong in general \cite{Schick(1998)}. 
In dimensions two and three analogues of the Gromov-Lawson-Rosenberg conjecture 
are true \cite{MOP}, but in dimension four there are additional obstructions coming 
from Seiberg-Witten theory. However, the following 
stable version of the conjecture conditionally holds in the following sense.

\begin{thm}[\cite{Stolz(2001)}] \label{Stolz} Assume that the real Baum-Connes assembly map $\mu_{\R}$
is injective 
for $\pi_1(M)$  and that $\alpha_{\R}(M) = 0$. Then some 
manifold of the form $M \times B^8 \times \ldots \times B^8$ admits 
a metric of positive scalar curvature, where $B^8$ is an arbitrary eight dimensional closed spin manifold with 
$\hat{A}(M) = 1$. 
\end{thm} 

This result is remarkable, because it is not understood how it can happen that a manifold $N$ does not 
admit a positive scalar curvature metric, but $N \times B^8$ does. Notice that the vanishing or 
non-vanishing of $\alpha_{\R}(M)$ is not affected, when $M$ is  multiplied with copies of $B^8$. In this 
respect Theorem \ref{Stolz} establishes $\alpha_{\R}(M)$ as the universal stable index theoretic obstruction 
to positive scalar curvature metrics.

If the assembly map for the maximal complex group $C^*$-algebra is injective, then also the rational assembly map 
\[
   K^G_*(\underline{E} G ) \otimes \Q = K_*(BG) \otimes \Q \to K_*(C^*_{max} G) \otimes \Q
\]
is injective. The strong Novikov conjecture \cite{Black} states that  here injectivity holds for all discrete groups $G$.

 Therefore it makes sense to single out those manifolds $M$ whose fundamental classes 
map nontrivially to $K_*(B \pi_1(M)) \otimes \Q$. This motivates the next definition.  

\begin{defn}  \label{Kessential}ÊA closed spin${}^c$ manifold $M^n$ is called {\em (rationally) $K$-theoretic essential}, if 
the classifying map $\phi : M \to B\pi_1(M)$ for the universal cover of $M$ satisfies 
\[
   \phi_*([M]_K) \neq 0 \in K_n(B\pi_1(M) ) \otimes \Q \, ,  
\]
where $[M]_K \in K_n(M)$ is  the $K$-theoretic fundamental class of $M$. 
\end{defn}

\begin{conj} \label{conjess} A $K$-theoretic essential spin manifold does not admit a metric of positive scalar curvature. 
\end{conj}

It follows from the previous remarks that this conjecture holds, if the rational assembly map for the associated 
fundamental group is injective. An important consequence of Conjecture \ref{conjess} is the following 

\begin{conj}[\cite{GL(1980b)}] Let $M$ be a closed aspherical spin manifold. Then $M$ does not 
admit a metric of positive scalar curvature. 
\end{conj}

The following is 
a variation of Definition \ref{Kessential} for singular homology. 

\begin{defn}[\cite{Gromov(1983)}] \label{homess} A closed oriented manifold $M^n$ is called {\em (rationally) essential}, if the classifying map $\phi : M \to B \pi_1(M)$ 
satisfies
\[
    \phi_*([M]_H) \neq 0 \in H_n(B \pi_1(M) ; \Q) \, , 
\]
where $[M]_H$ is the fundamental class of $M$ in singular homology. 
\end{defn} 

Recall that the homological Chern character defines an isomorphism 
\[
   \ch : K_{(*)}(M) \otimes \Q \cong H_{(*)}(M ; \Q) \, , 
\]
where the brackets in the subscripts indicate that we regard both 
theories as $\Z/2$-graded. Keeping in mind that for a closed spin${}^c$ manifold $M^n$ we have 
\[
   \ch([M]_K)  = [M]_H + c 
\]
where $c \in H_{< n }(M ; \Q)$ we see that essential spin${}^c$ manifolds are also $K$-theoretic essential. 
Hence it makes sense to formulate the following conjecture.

\begin{conj} An essential manifold does not admit a metric of positive scalar curvature. 
\end{conj}

This seems especially intriguing, if the universal cover of  this manifold is not spin 
(so that index theoretic obstructions 
are not available). Evidence for the 
conjecture in this case is provided by the fact that sometimes essential manifolds satisfy a weak form 
of enlargeability \cite{Dranish1, Dranish2}. 


\section{$K$-area for Hilbert module bundles}  \label{Karea} 


All manifolds in this section are closed, smooth and connected. 
We recall the following definition from  \cite{GL(1983)}. 

\begin{defn} \label{enlargeable}ÊLet $(M^n,g)$ be an orientable Riemannian manifold. 
\begin{itemize} 
   \item We call $M$ {\em enlargeable}, if for every $\epsilon > 0$ there is a Riemannian 
             cover $(\overline{M}, \overline{g})$ of $(M,g)$ together 
with an $\epsilon$-Lipschitz map $f_{\epsilon} : \overline{M} \to S^n$ which is constant outside of a compact subset of 
       $\overline M$ and of non-zero degree. 
\item We call $(M,g)$ {\em area-enlargeable}, if for every $\epsilon > 0$ there is a Riemannian 
cover $(\overline{M}, \overline{g})$ of $(M,g)$ together with  a smooth map $f_{\epsilon} : \overline{M} \to S^n$ 
which is    $\epsilon$-contracting on $2$-forms, constant outside 
of a compact subset of $\overline M$  and  of nonzero degree.
\end{itemize}
\end{defn}

Because $M$ is compact, all Riemannian metrics on $M$ are in bi-Lipschitz correspondence and hence
both of the above properties are independent of the specific choice of the metric $g$ on 
$M$. Enlargeability is therefore  a purely topological property of $M$. Indeed, whether $M$ is 
enlargeable depends only on the 
image of the fundamental class of $M$ in the rational group homology of $\pi_1(M)$ under the classifying 
map, see \cite[Corollary 3.5]{BrunnHan} restated as  Theorem \ref{BrunnH} below.
We do not know whether a similiar result holds for area-enlargeability. 

Examples for enlargeable manifolds are manifolds which admit Riemannian metrics of 
nonpositive sectional curvature. This follows from the Cartan-Hadamard theorem.

Area-enlargeable spin manifolds allow the construction of finite dimensional Hermitian 
twisting bundles for the Dirac operator as described after Example \ref{badex}. 
We remark that 
the index theoretic setting explained there needs to be slightly 
generalized (relative index theory on open manifolds, see \cite{GL(1983)}), if infinite covers of $M$ are involved (this case is not excluded in 
Definition \ref{enlargeable}). These considerations lead to the following theorem.

\begin{thm}[\cite{GL(1980b), GL(1983)}] \label{gromov-lawson} 
Let $M$ be an area-enlargeable spin manifold. Then $M$ does not admit a metric of 
positive scalar curvature. 
\end{thm} 

At this point one might ask whether the enlargeability obstruction 
is reflected by the Rosenberg obstruction. 

The twisting bundles $E \to M$ of arbitrarily small curvature 
going into the obstruction expressed in Theorem \ref{gromov-lawson} 
motivate the notion of {\em $K$-area}, see \cite{Gromov(1995)}. 

In this section we will introduce a related property for $K$-homology classes of $M$. 
Examples of such $K$-homology classes are $K$-theoretic fundamental classes of 
area-enlargeable spin manifolds, see Proposition \ref{enlinfinite}.
The main result in this section, Theorem \ref{Kess}, 
shows that classes in $K_0(M) \otimes \Q$ of infinite $K$-area are mapped to non-zero classes 
in $K_0(C^*_{max} \pi_1(M))$ under the assembly map. Together with Proposition \ref{enlinfinite} this  implies that 
 the Rosenberg obstruction subsumes the enlargeability obstruction of Gromov and Lawson:

\begin{thm}[\cite{HS(2006), HS(2007)}] Let $M^n$ be an area-enlargeable spin manifold. Then the Rosenberg index 
$\alpha(M) \in  K_n(C^*_{max} \pi_1(M))$ is different from zero. 
\end{thm} 

A convenient setting for our discussion is provided by Kasparov's $\KK$-theory, cf. \cite{Black}, which associates
to any pair of separable $C^*$-algebras $A$ and $B$ an abelian group $\KK(A,B)$. We 
work over the field of complex numbers and will restrict attention to the special cases $A = C(M)$, $B = \C$ and $A = \C$, 
$B = C(M) \otimes S$ for 
a seperable unital $C^*$-algebra $S$. Here we will work only with ungraded $\KK$-groups. 

According to the analytic description of $K$-homology \cite{HR} we have a canonical identification 
\[
   \KK(C(M), \C) \cong K_0(M)
\]
the $0$-th $K$-homology of $M$ which, for example, can be defined homotopy theoretically as the homology theory
dual to topological $K$-theory \cite{Atiyah}. 

Elements in $\KK(A,B)$ are represented by  {\em Fredholm triples} $(E, \phi, F)$
where $E$ is a countably generated graded Hilbert $B$-module, $\phi: A \to \mathcal{B}(E)$ is a graded 
$*$-homomorphism (here $\mathcal{B}(E)$ is the graded $C^*$-algebra of graded adjointable
bounded  $B$-module 
homomorphisms $E \to E$) and $F \in \mathcal{B}(E)$ is an operator of degree $1$ such that the commutator
$[F, \phi(a)]$ and the operators $(F^2 - \id_E)\phi(a)$ and $(F-F^*) \phi(a)$ are $B$-compact  for all $a \in A$. 
In our context we will be dealing with Fredholm triples of very special forms which will be specified in a moment.  The reader who is interested in more information on the notion of Hilbert modules 
and the construction of Kasparov $\KK$-theory can consult the sources \cite{Black,Wegge-Olsen}. 

A typical situation arises  when $M$ is a spin manifold of even dimension equipped with a Riemannian metric $g$.  The Dirac operator from Section \ref{index} 
\[
   D_g : \Gamma(S^{\pm}) \to \Gamma(S^{\mp}) 
\]
is a symmetric graded first-order elliptic differential operator. It therefore gives rise 
to an element $[D_g] \in \KK(C(M), \C)$ represented by the 
 Fredholm triple $(L^2(S) , \phi, F)$ where $L^2(S)$ is the space of $L^2$-sections of the bundle $S^+ \oplus S^-$, the map 
 $\phi: C(M) \to \mathcal{B}(L^2(S))$ is the standard representation as multiplication operators and $F \in \mathcal{B}(E)$  
 is a bounded operator which is obtained from $D_g$ by functional calculus. 

The construction works more generally for symmetric graded elliptic differential operators on graded smooth Hermitian vector bundles over $M$, cf. \cite[Theorem 10.6.5]{HR}. In this way we may 
think of elements in $\KK(C(M), \C) = K_0(M)$ as a kind of generalized symmetric elliptic differential operators over $M$. 
In this picture the index  of a graded elliptic differential operator corresponds to the image of the 
$\KK$-class represented by this operator under the map 
\[
    K_0(M) \to K_0(*)  = \Z 
\]
which is induced by the unique map $M \to *$. 

If $E \to M$ is a (finite dimensional) Hermitian bundle with a Hermitian connection we obtain the twisted Dirac 
operator 
\[
    D_{g,E} : \Gamma(S^{\pm} \otimes E) \to \Gamma(S^{\mp} \otimes E) 
\]
which is again a symmetric graded  elliptic differential operator and has an index $\ind(D_{g,E}) \in \Z$. 

The index 
of the twisted operator $D_{g,E}$ has the following description in $\KK$-theory, cf.~\cite{Black}.  
The bundle $E \to M$ represents a class $[E]$ in topological $K$-theory $K^0(M)$, which can be canonically 
identified with $\KK(\C, C(M))$. The element $[E] \in \KK(\C, C(M))$ is represented by the Fredholm triple
$(\Gamma(E), \phi, 0)$ where $\Gamma(E)$ is the $C(M)$-module of continuous sections
$M \to E$ equipped with the $C(M)$-valued inner product given by fibrewise application of the Hermitian 
inner product on $E$ and $\phi:  \C \hookrightarrow \mathcal{B}(\Gamma(E))$ is the standard embedding. 

Under the Kasparov product map \cite{Black} 
\[
     \KK(\C, C(M)) \times \KK(C(M), \C) \to \KK(\C, \C) = \Z
\]
which in this case corresponds to the usual Kronecker product pairing of $K$-homology and topological $K$-theory 
(i.e. $K$-cohomology)
\begin{eqnarray*} 
    K^0(M) \times K_0(M) & \to & \Z  \\
       (c,h) & \mapsto & \langle c , h \rangle 
\end{eqnarray*}
the pair $([E], [D_g])$ is sent to $\ind(D_{g,E}) \in \Z$. 

This point of view may be generalized by allowing twisting bundles 
$E \to M$ which are locally trivial  bundles of finitely generated right Hilbert $A$-modules where $A$ is a 
unital $C^*$-algebra. 

We recall \cite{Schick(2005), Wegge-Olsen} 
that each finitely generated Hilbert $A$-module bundle $E \to M$ is isomorphic 
to an orthogonal direct summand of a trivial $A$-module bundle $M \times A^n \to M$ where $A^n$ 
carries the canonical $A$-valued inner product
\[
   \langle   (a_1, \ldots, a_n) , (b_1, \ldots, b_n) \rangle \mapsto a_1^* b_1 + \ldots + a_n^* b_n \, . 
\]
We can take this description as definition of finitely generated Hilbert $A$-module bundles. 

Let $E \to M$ be a finitely generated Hilbert $A$-module 
bundle.  We associate to $E \to M$ a $\KK$-class $[E] \in KK(\C , C(M) \otimes A)$ as follows. 
First note that the space $\Gamma(E)$ of continuous sections in $E$ is a finitely 
generated Hilbert $(C(M) \otimes A)$-module and the identity $\Gamma(E) \to \Gamma(E)$
is a $(C(M) \otimes A)$-compact 
(indeed finite rank) operator by a partition of unity argument. Therefore the triple $(\Gamma(E), \phi, 0)$, 
where $\phi: \C \to \mathcal{B}(\Gamma(E))$ is the 
standard embedding, defines an element in $\KK(\C , C(M) \otimes A)$. 

Using the Kasparov product (which we again interprete as a Kronecker product pairing) 
\begin{eqnarray*}
    \KK(\C, C(M) \otimes A) \times \KK(C(M), \C) & \to & \KK(\C, A) = K_0(A) \\
       (c,h)     & \mapsto  & \langle c , h \rangle 
\end{eqnarray*}
we have a pairing of generalized elliptic differential operators on $M$ and finitely 
generated Hilbert $A$-module bundles. 

If $M$ is a Riemannian spin manifold of even dimension, then the element in $\langle [E], [D_g] \rangle \in 
K_0(A)$ can be interpreted as the index of the Dirac 
operator $D_g$ twisted with the bundle $E$, cf.~\cite{Black}.  
Hence, for the special case when $E \to M$ is the Mishchenko-Fomenko bundle, the class 
$\langle [E], [D_g] \rangle$
coincides with the Rosenberg index $\alpha(M)$ defined in Section \ref{index}.

We will now single out those $K$-homology classes $h \in K_0(M)$ which 
can be detected by finitely generated Hilbert $A$-module bundles of arbitrarily small curvature. In the following 
let $M$ be a closed smooth Riemannian manifold. In order
to avoid the discussion of smooth bundles and curvature notions for infinite dimensional bundles we
proceed as follows.

Recall that the {\em path groupoid} $\mathcal{P}_1(M)$ of $M$  has as objects the points in $M$ and 
as morphisms $\mathcal{P}_1(M)(x,y)$ the set of piecewise smooth paths $[0,1] \to M$
connecting $x$ and $y$. This is a topological category, in particular both the sets of objects and morphisms are
topological spaces. 

Let $A$ be a unital $C^*$-algebra and let $E \to M$ be a finitely generated
Hilbert $A$-module bundle.  
The {\em transport category} $\mathcal{T}(E)$ 
has as objects the points in $M$ and as set of morphisms 
\[
    \mathcal{T}(E)(x,y) := {\rm Iso}_A(E_x,E_y) \,.  
\]
This is again a topological category where the set of morphisms is topologized by choosing 
local trivializations in order to identify nearby fibres of $E \to M$ and the set of Hilbert 
$A$-module isomorphisms 
${\rm Iso}_A(E_x,E_y)$ is topologized as a subset of the Banach space $\Hom_A(E_x,E_y)$. 

A {\em holonomy representation} on $E \to M$ is a continuous functor
\[ 
     \mathcal{H}: \mathcal{P}_1(M) \to \mathcal{T}(E) \, . 
\]
It is called {\em $\epsilon$-close to the identity at scale $\ell$}, if for each $x \in M$ and 
each closed loop $\gamma \in \Mor
(\mathcal{P}_1(M))$ based at $x \in M$ and of length $\ell(\gamma) \leq \ell$ we have
\[
  \|  \mathcal{H}(\gamma) - \id_{E_x} \| < \epsilon \cdot \ell( \gamma) \, . 
\]
Here we use the operator norm on the left hand side.

The following proposition establishes a link to the notion to parallel transport in differential geometry.

\begin{prop} \label{altSchick} Depending on $M^n$ there are a real constants $C, \ell > 0$ 
so that the following holds.
Let $E \to M$ be a finite dimensional smooth Hermitian bundle of rank $d$ equipped with a 
smooth Hermitian connection $\nabla$ 
whose curvature $\eta \in \Omega^2(M ; \mathfrak{u}(d))$ is norm bounded by $\epsilon$.
Then the parallel transport with respect to $\nabla$ 
is $(C  \cdot \epsilon)$-close to the identity at scale $\ell$. 
\end{prop} 

\begin{proof} By a Lebesgue number argument there is a small $\ell > 0$ and a cover of $M^n$ 
by finitely many closed subsets $D_1, \ldots, D_k \subset M$ so that the following holds: Each $D_i$ is diffeomorphic 
to the $n$-dimensional unit cube $[0,1]^n \subset \R^n$ and each closed loop in $M$ of length 
at most $\ell$ is contained in a subset $D_i$. It is hence enough to prove the assertion 
for a closed loop $\gamma \in \Mor(\mathcal{P}_1(M))$ contained in one of these subsets $D_i \subset M$ and based at 
a point $x \in D_i$. In the following we write $D$ instead of $D_i$ and identify $D$ and $[0,1]^n$ by 
a fixed diffeomorphism. 

Let $E \to M$ be a Hermitian bundle of rank $d$ as described in the proposition. 
We construct a trivialization of $E|_D \to M$ by choosing an isomorphism
$E|_{(0, \ldots, 0)} \cong \C^d$ and extending the trivialization inductively into each 
of the $n$ coordinate directions by parallel transport.  We denote the induced connection one form with 
respect to this trivialization by $\omega \in \Omega^1 (D; \mathfrak{u}(d))$. 

Now an argument 
similiar to \cite[Lemma 2.3]{HS(2006)}, but using the Riemannian metric on $[0,1]^n$ induced by $M$, 
shows that there is a number $C > 0$, which depends on $D$, but not on the bundle $E \to M$, so that 
\[
    \| \omega|_D \| \leq  C \cdot   \| \eta|_D \|  \, , 
\]
where we use  the operator norm on $\mathfrak{u}(d)$ and 
the maximum norms on the unit sphere bundles of $T^*D$ and $\Lambda^2 D$.

Let $\phi : [0,1] \to E$ be a parallel vector field along a piecewise smooth (not necessarily closed) 
path $\zeta : [0,1] \to D \subset M$. By virtue of the given trivialization
consider $\phi$ as a smooth map $[0,1] \to \C^d$. As such it satisfies the differential equation 
\[
    \phi'(t) + (\omega_{\gamma(t)} (\gamma'(t))) \cdot \phi(t) = 0 
\]
and it follows that
\[
    \| \phi(1) - \phi(0) \| \leq \exp \big( \ell (\zeta) \cdot \| \omega|_D \| \big) \cdot \| \phi(0) \| \, . 
\]
Because we started with a Hermitian connection on $E$ we get $\|\phi(1) \| = \|\phi(0)\|$ which implies that 
we can assume (by subdividing $\zeta$ into small pieces and appealing to the triangle 
inequality) that $\ell(\zeta)$ is arbitrarily small. Because $\exp :  \C^d \to \C^d$ is uniformly Lipschitz 
continuous on each bounded neighbourhood of $0$ with Lipschitz constant arbitrarily close to $1$
we hence obtain
\[
   \| \phi(1) - \phi(0) \| \leq 1.5 \cdot \ell (\zeta) \cdot \| \omega|_D \| \cdot \| \phi(0) \| 
 \]
 from which the claim of the proposition follows. 
\end{proof} 

\begin{defn} \label{Kflach}ÊLet $M$ be a closed smooth manifold and let $h \in K_0(M) \otimes \Q$. 
We say that $h$ has {\em infinite $K$-area}, if there is a Riemannian metric on $M$ and a number $\ell > 0$ 
so that the following holds: For each $\epsilon > 0$ there is a unital $C^*$-algebra 
$A$ and a finitely generated Hilbert $A$-module bundle $E \to M$  which 
carries a holonomy representation which is $\epsilon$-close to the identity at scale $\ell$ and 
satisfies 
\[
      \langle [E] , h \rangle \neq 0 \in K_0(A) \otimes \Q
\]
where $[E] \in \KK(\C, C(M) \otimes A)$ is the element represented by $E \to M$. If $h$ is not of infinite $K$-area, we 
say that it is {\em of finite $K$-area}. 

A class $h \in H_{ev}(M;\Q)$ is defined to be of infinite $K$-area, if the class $\ch^{-1}(h) \in K_0(M) \otimes \Q$ is of infinite $K$-area. 
\end{defn}

By adapting the involved scale appropriately it is clear that for testing whether $h$ is of infinite $K$-area or not 
any Riemannian metric on $M$ can be used. 

The notion of finitely generated Hilbert $A$-module bundles can be generalized
to $C^*$-algebras without unit. However, in the context of Definition \ref{Kflach}, 
this does not result in a wider class of $K$-homology classes of infinite $K$-area, since 
any finitely generated Hilbert $A$-module bundle is in a trivial way also a finitely 
generated Hilbert $A^+$-module bundle over the unitalization $A^+$ of $A$. This procedure 
does not change the property of $\langle [E], h \rangle $ being zero or not (in the rationalization of the 
$K$-homology of $A$ and $A^+$ respectively). 

Our Definition \ref{Kflach} is inspired by the preprint  \cite{Listing(2010)} 
where the property of finite $K$-area is investigated from a homological perspective. 
In contrast to the approach in loc.~cit.~and in the original source \cite{Gromov(1995)} we do not further quantify 
classes of finite $K$-area, since we will be concentrating on the property of infinite $K$-area as one instance of 
a largeness property besides enlargeability and essentialness.   
The discussion in \cite{Listing(2010)} and other previous papers
is restricted to finite dimensional smooth Hermitian vector bundles as twisting bundles $E \to M$ occuring in 
our Definition \ref{Kflach}. 
Our more general setting is needed in connection with 
enlargeability questions and applications to the strong Novikov conjecture, see Section \ref{Novikov}. 

By a suspension procedure we can also define classes in $h \in K_1(M) \otimes \Q$ of infinite $K$-area by 
requiring that the class $h \times [S^1]_K \in K_0(M \times S^1) \otimes \Q$ be of infinite $K$-area, with an 
arbitrary choice of  a $K$-theoretic 
fundamental class $[S^1]_K \in K_1(S^1)$. Note that with this definition the class $[S^1]_K \in K_1(S^1) \otimes \Q$ is 
of infinite $K$-area. The following discussion can be extended to $K$-homology classes of 
odd degree, but we restrict our exposition to classes in $K_0(M) \otimes \Q$ for simplicity.  

The following two facts are similiar to Propositions 2 and 3 in \cite{Listing(2010)}, cf.~also Proposition 3.4. and Theorem 3.6 in \cite{BrunnHan}. 

\begin{prop} \label{subvector} The elements of finite $K$-area in $K_0(M) \otimes \Q$ form a rational vector subspace. 
\end{prop} 

\begin{proof} Obviously $0 \in K_0(M) \otimes \Q$ is of finite $K$-area. 
If $h \in K_0(M) \otimes \Q$ is of infinite $K$-area, then the same is true for any nonzero rational 
multiple of $h$. This implies that the set of elements of finite $K$-area is closed under scalar multiplication. 
Now assume that $h + h'$ is of infinite $K$-area.  It follows from Definition \ref{Kflach} that either $h$ or $h'$ 
are of infinite $K$-area (choose $\epsilon := \frac{1}{k}$ with $k = 1,2, \ldots$). This shows that the set of elements of 
finite $K$-area is closed under addition. 
\end{proof} 

\begin{prop} If $f : M \to M'$ is a continuous map, then $f_* : K_0(M) \otimes \Q \to K_0(M') \otimes \Q$ 
restricts to a map between vector subspaces consisting of elements of finite $K$-area. In particular, the vector subspace of 
elements of finite $K$-area in $K_0(M) \otimes \Q$ is an invariant of the homotopy type of $M$. 
\end{prop} 

We will return to homological aspects of largeness properties in Section \ref{large_homology}. The 
notion of infinite $K$-area is illustrated by the following examples. 

Assume that $M$ is an oriented manifold of even dimension $2n$ which has infinite $K$-area in the sense of Gromov \cite{Gromov(1995)}. By definition this
means that for each $\epsilon > 0$ there is a finite dimensional smooth Hermitian vector bundle 
$V \to M$  with a Hermitian connection whose curvature form in  $\Omega^2(M; \mathfrak{u}(d))$ (where $d = \rk V$)
has norm smaller than $\epsilon$ and with at least one nonvanishing Chern number. 

Using linear combinations of tensor products and exterior products of $V$ one can show that there is a 
Hermitian bundle $E \to M$ with Hermitian connection whose curvature has norm smaller 
than $C \cdot \epsilon$ (where $C$ is a bound which depends only on $\dim M$) and which satisfies 
\[
    \langle \ch(E)  , {\rm PD}(\mathcal{\hat{A}}(M))  \rangle  \neq 0 \in H_0(M ;\Q)   \, , 
\] 
where ${\rm PD}(\mathcal{\hat{A}}(M))$ is the Poincar\'e dual in $H_{ev}(M ; \Q)$ 
of the $\mathcal{\hat{A}}$-polynomial of $M$.

The precise argument is carried out in \cite{Dav} where the following fact is shown. There is a number 
$N$ depending only on $\dim M$ with the following property: Assume that 
$V \to M$ is a complex vector bundle and assume that 
all bundles $V' \to M$ which may be constructed out of $V$ by at most $N$ operations of the form 
direct sum, tensor product and exterior product satisfy
\[
    \langle \ch(V'),  {\rm PD}(\mathcal{\hat{A}}(M))  \rangle  = 0 \in H_0(M ;\Q) \, . 
\] 
Then all Chern numbers of $V\to M$ are zero.

Considering Hermitian vector bundles as finitely generated Hilbert $\C$-module 
bundles this means in the language of Definition \ref{Kflach} that the class ${\rm PD}(\mathcal{\hat{A}}(M))  \in H_{ev}(M;\Q) $ 
has infinite $K$-area (here we use that the Chern character is compatible with the Kronecker pairing). 
If $M$ is equipped with a spin structure, this element is equal to $\ch([M]_K)$, 
the Chern character applied to the $K$-theoretic 
fundamental class of $M$,  and hence we have shown that under the stated assumptions the class 
$[M]_K$ has infinite $K$-area in our sense. 

By a similar argument one shows that if $M$ has infinite $K$-area in the sense of Gromov, then 
\[
      [M]_H  \in H_{2n}(M;\Q)
\]
has infinite $K$-area, where $[M]_H \in H_{2n}(M ; \Q)$ is the homological fundamental class of $M$.

As a second example, cf.~\cite[Section 4]{HS(2006)}, assume that $M$ is area-enlargeable and that the covers $\overline{M} \to M$ 
in Definition \ref{enlargeable} can always be assumed to be finite. By pulling back a suitable Hermitian bundle 
$V \to S^{2n}$ with connection to 
these covers along the maps $f_{\epsilon} : \overline{M} \to S^{2n}$ and wrapping these 
bundles up to get finite dimensional Hermitian bundles $E \to M$ with small curvature, one 
can show that the classes $[M]_H \in H_{2n}(M;\Q)$ and $[M]_K\in K_0(M) \otimes \Q$ (if $M$ is spin) 
have infinite $K$-area. 

More generally assume that $M^{2n}$ is area-enlargeable with no restriction on the covers 
$\overline{M} \to M$. Then \cite[Proposition 1.5]{HS(2007)} implies that 
the classes $[M]_H$ and $[M]_K$, respectively, have infinite $K$-area. In this case we need infinite dimensional 
bundles $E \to M$ which shows the usefulness of  Definition \ref{Kflach} in the general context 
of Hilbert $A$-module bundles where $A$ is a $C^*$-algebra different from $\C$. 

For later reference we state the last observation seperately.

\begin{prop} \label{enlinfinite}Ê Let $M$ be area-enlargeable and of even dimension. 
Then the $K$-area of $[M]_H$ is 
infinite. If $M$ is equipped with a  spin structure, then also the $K$-area of $[M]_K$ is infinite.  
\end{prop} 

We denote by 
\[
   \alpha : K_0(M) \to K_0(B \pi_1(M))   \stackrel{\mu}{\to}  K_0(C^*_{max} \pi_1(M)) 
\]
the composition of the map induced by the classifying map $M \to B \pi_1(M)$ and 
the assembly map. If $M$ is a spin manifold of even dimension, 
note the equations 
\[
   \alpha(M) = \alpha([M]_K)  
\]
(the left hand side coincides with the Rosenberg index) and - more generally - 
\[
    \alpha(h) =  \langle [E] , h \rangle  \in K_0(C^*_{max} \pi_1(M)) \otimes \Q 
\]
for all $h \in K_0(M) \otimes \Q$ where $E \to M$ is the Mishchenko-Fomenko bundle for $C^*_{max} \pi_1(M)$.

The following is the main result of our paper. 

\begin{thm} \label{Kess} Let $M$ be a closed connected smooth manifold and let 
$h \in K_0(M)\otimes \Q$ be of infinite $K$-area. Then 
\[
    \alpha(h) \neq 0 \in K_0(C^*_{max} \pi_1(X)) \otimes \Q \, . 
\]
\end{thm}

We note the following implication for the Rosenberg index. 

\begin{cor} \label{cor1}ÊLet $M$ be a closed spin manifold of even dimension whose $K$-theoretic fundamental class has 
infinite $K$-area. Then 
\[
   \alpha(M) \neq 0 \in K_0(C^*_{max} \pi_1(M)) \, . 
\]
In particular, closed even-dimensional 
spin manifolds of infinite $K$-area in the sense of Gromov \cite{Gromov(1995)} have
nonvanishing Rosenberg index. (A similar result holds, if $M$ is odd dimensional.) 
\end{cor}

The proof of Theorem \ref{Kess} is based on the construction of ``infinite product bundles'' from 
\cite{HS(2006)}. We shall explain how this construction fits the setting of the paper at hand.  

Let $(E_k)_{k \in \N}$ be a sequence of finitely generated Hilbert $A_k$-module bundles over $M$, where 
$(A_k)$ is a sequence of unital $C^*$-algebras.  We assume 
that the fibre of $E_k$ is isomorphic (as a Hilbert $A_k$-module) to $q_k A_k$ where 
$q_k \in A_k$ is a (self-adjoint) projection. This assumption is important for our construction. In 
general the fibre of $E_k$ is of the form $q \cdot (A_k)^n$ for some $n$ with a projection 
$q \in \Mat(A_k,n)$. In this case we use the same transition functions as for  $E_k$ to construct 
a Hilbert $\Mat(A_k,n)$-module bundle of the required form. By Morita equivalence of 
$A_k$ and $\Mat(A_k,n)$ this does not affect the
 $K$-theoretic considerations relevant for our discussion. 
 
We consider the unital $C^*$-algebra $A$ consisting of norm bounded sequences 
\[    
    (a_k)_{k \in \N} \in \prod_{k=1}^{\infty} A_k   
\]
and wish to construct a Hilbert $A$-module bundle $E \to M$ with fibre $q A$, where $q = (q_k)$ is 
the product of the projections $q_k$, by taking the ``infinite product'' of the bundles $E_k$.  
However, taking the infinite product of the transition functions for the bundles $E_k$ may not 
result in continuous transition functions for the infinite product bundle. 
The following example indeed shows that 
an infinite product construction of this kind may be obstructed by topological properties of the bundles $E_k$.

\begin{ex} Let $E_k \to S^2$ be the complex line bundle with Chern number $k$. 
Assume we have a Hilbert $A$-module bundle $E \to S^2$ over the 
$C^*$-algebra $A = \prod_k  \C$ (which is equal to the standard seperable Hilbert space) 
with typical fibre $V = \prod_k \C$ and Lipschitz continous transition functions in diagonal form so that 
the $k$th component of this bundle is isomorphic to $E_k$ as a complex line bundle. 

Restricting the transition functions  of $E$  to the single factors leads 
to trivializations for the bundles $E_k \to S^2$ whose transition functions have uniformly (in $k)$ 
bounded  Lipschitz constants. This implies that the Euler 
numbers of the bundles $E_k$ are bounded, contrary to our assumption. 
\end{ex}

This example indicates that we need to choose Lipschitz trivializations of the bundles $E_k$ so that the resulting transition functions have uniformly bounded Lipschitz constants. This can be achieved as follows. 

\begin{prop} \label{prod} Assume that each bundle $E_k\to M$ is equipped with a 
holonomy representation $\mathcal{H}_k$ so that $\mathcal{H}_k$ is $\epsilon$-close to the identity at scale 
$\ell$ where the constants $\epsilon$ and $\ell$ are independent of $k$, and $M$ is equipped with a 
fixed Riemannian metric.  
Then there is a finitely generated Hilbert $A$-module bundle 
$V \to M$ with transition functions in diagonal form and so that the $k$th component of this 
bundle is isomorphic to $E_k$ as an $A_k$-Hilbert module bundle.
\end{prop} 

\begin{proof} We start with a cover of $M^n$ by finitely many closed subsets $(D_i)_{i \in I}$ each of which 
is diffeomorphic to the $n$-dimensional unit cube $[0,1]^n \subset \R^n$ and so that the interiors of these subsets 
still cover $M$. The size of each $D_i$ can be assumed to be small compared to $\ell$. 

For each $k$, using the holonomy representation $\mathcal{H}_k$, we trivialize the bundle $E_k$ over each subset $D_i$ 
inductively into each of the $n$ coordinate directions (compare the proof of Proposition \ref{altSchick}).  

This leads to local trivializations of $E_k|_{D_i}$ 
whose transition maps (for fixed $k$, but varying $i$) have uniformly bounded (in $i$ and $k$)
Lipschitz constants. Hence the 
product of these transition maps  can be used to define the Hilbert $A$-module bundle $V \to M$ as 
required. 
\end{proof} 

We remark that the product bundle $V \to M$ is a bundle of finitely generated Hilbert 
$A$-modules isomorphic to $qA$ by our assumption that $E_k$ has typical fibre $q_k A_k$. 

For the proof of Theorem \ref{Kess} we assume that $h \in K_0(M) \otimes \Q$ and that $(E_k)$ is a sequence of 
Hilbert $A_k$-module bundles with fibres $q_k A_k$ so that $\langle [E_k] , h \rangle  \neq 0 \in K_0(A_k) \otimes \Q$ for all $k$. Furthermore, 
we assume that $E_k$ is equipped with a holonomy representation $\mathcal{H}_k$ 
which is $1/k$-close to the identity at some scale $\ell$ which is independent of $k$.  

We consider the Hilbert $A$-module bundle $V \to M$ constructed in Proposition \ref{prod}. 

Starting from  $V$ we can construct various other Hilbert module bundles over $M$ as follows. Let 
\[
   \psi_k : A \to A_k 
\]
denote the projection onto the $k$th component. Moreover, we denote by 
\[
   A' := \bigoplus_{k=1}^{\infty} A_k \subset A  
\]
the closed two sided ideal consisting of sequences in $A$ tending to zero and by 
\[
   Q := A / A' 
\]
the quotient $C^*$-algebra. Finally, let 
\[
   \psi : A \to Q 
\]
be the quotient map. 

We obtain Hilbert $A_k$-bundle isomorphisms
\[
     E_k \cong  V \otimes A_k 
\]
and a Hilbert $Q$-module bundle
\[
    W := V \otimes Q 
\]
with typical fibre $q Q$, where we identify $q \in A$ and its image in $Q$.  

The following fact is crucial

\begin{prop} 
The bundle $W$ has local trivializations with locally constant transition
maps. More precisely, it can be written 
as an associated bundle 
\[
      W = \widetilde{M} \times_{\pi_1(M)} qQ
\]
for some unitary representation $\pi_1(M)  \to \Hom_Q(qQ, qQ)$. 
\end{prop}

\begin{proof} The family of holonomy representations $(\mathcal{H}_k)$ induces a holonomy 
representation on $W$ which is equal to the identity on each closed loop of length at most $\ell$ in $M$ (and hence 
on contractible loops of arbitrary length), 
because the holonomy representation $\mathcal{H}_k$ is $1/k$-close to the identity at scale $\ell$. 
Using this holonomy representation on $W$ we construct the desired local trivializations of $W$. 
\end{proof}

These facts in combination with naturality properties of Kasparov $\KK$-theory
allow us to show 
that $\alpha(h) \neq 0 \in K_0(C^*_{max} \pi_1(M)) \otimes \Q$. The holonomy representation for the 
bundle $W$ induces an involutive map 
\[
   \pi_1(M) \to \Hom_Q(qQ, qQ) = q Q q
\]
with values in the unitaries of the $C^*$-algebra $q Q q$. Hence, by the universal property 
of $C^*_{max} \pi_1(M)$ we get an induced map of $C^*$-algebras
\[
   \phi :  C^*_{max} \pi_1(M) \to q Q q \hookrightarrow Q \, . 
\]
Note that this step is not possible in general, if we use the reduced $C^*$-algebra 
$C^*_{red} \pi_1(M)$ instead. 
Let $E = \widetilde{M} \times_{\pi_1(M)} C^{*}_{max} \pi_1(M) \to M$ 
be the Mishchenko-Fomenko bundle. 

We study the commutative diagram 
\[
   \xymatrix{
  K_0(M) \ar[d]^{=} \ar[r]^-{\langle [E], - \rangle} & K_0(C^*_{max} \pi_1(M) ) \ar[r]^-{\phi_*} & K_0(Q) \ar[d]^{=} \\
  K_0(M)  \ar[r]^{\langle [V], - \rangle }                 & K_0(A)  \ar[r]^{\psi_*}               & K_0(Q )           \\
        }
\]
The composition 
\[
   K_0(M) \stackrel{\langle [V] , - \rangle }{\longrightarrow} K_0(A) \stackrel{(\psi_k)_*}{\longrightarrow} K_0(A_k) 
\]
sends the element $h$ to $\langle [E_k], h \rangle  \in K_0(A_k)$ which is different from 
zero by assumption. This implies that under the map 
\begin{eqnarray*}
  \chi :  K_0(A) & \to & \prod_k K_0(A_k ) \\
        z &  \mapsto & ( (\psi_k)_*(z))_{k  = 1,2, \ldots} 
\end{eqnarray*}
the element $z : = \langle [V], h \rangle$ is sent to a sequence all of whose components are 
different from zero. We will conclude from this that also $\psi_*(z) \neq 0$ finishing 
the proof of Theorem \ref{Kess}.

Consider the long exact sequence in $K$-theory induced by the short exact sequence
\[
    0 \to A' \to A \to Q \to 0 \, . 
\]
Using the fact that $K$-theory commutes with direct limits we have a canonical 
isomorphism 
\[
    K_0(A') \cong \bigoplus_k K_0(A_k) \, . 
\]
Assume that $\psi_*(z) = 0$. This implies that $\chi$ maps $z$ 
to a sequence $(z_k) \in \prod_k K_0(A_k)$ 
with only finitely many nonzero entries. But this contradicts the calculation that we carried out 
before. Hence $\psi_*(z) \neq 0$.


\section{The strong Novikov conjecture} \label{Novikov} 


The method presented in the previous paragraph can be used to prove a special case of the 
strong Novikov conjecture. Let $G$ be a discrete group and  let $\Lambda^*(G) \subset H^*(BG; \Q)$ be 
the subring generated by $H^{\leq 2}(BG; \Q)$

\begin{thm}[\cite{HS(2008)}]  \label{strnovikov}ÊLet $h \in K_0(BG)\otimes \Q$ be a $K$-homology class with the following 
property: There is a class $c \in \Lambda^*(G)$ so that $\langle c , \ch(h) \rangle  \neq 0 \in H_0(BG;\Q) = \Q $. Then under the assembly map 
\[
    K_0(BG) \otimes  \Q  \to K_0(C^*_{max} G) \otimes \Q 
 \]
the element $h$ is sent to a a non-zero class. 
\end{thm}
As a corollary one obtains the following special case of the classical Novikov conjecture. 

\begin{cor}[\cite{CGM,Mathai}] Let $M$ be a connected closed oriented manifold, let $G$ be a discrete group and let
$f: M \to BG$ be a continuous map. Then for all $c \in \Lambda^*(G)$ the higher signature $\langle \mathcal{L}(M)  \cup f^*(c) , [M] \rangle$ 
is an oriented homotopy invariant, where $\mathcal{L}(M)$ denotes the Hirzebruch $L$-polynomial. 
\end{cor} 

We will establish Theorem \ref{strnovikov} as a fairly straightforward consequence of Theorem \ref{Kess}. It illustrates 
again the flexibility of the notion of 
infinite $K$-area in Definition \ref{Kflach} based on Hilbert module bundles. 
For simplicity we restrict to the case when there is a class $c \in H^2(BG; \Q)$
with $\langle c, \ch(h) \rangle \neq 0$. Furthermore, without loss of generality, we can assume 
that $G$ is finitely presented. The general case follows by applying a direct limit argument.  

Using the description of $K$-homology due to Baum and Douglas \cite{BD}  there is a closed connected spin manifold $M$ 
of even dimension (which can be chosen arbitrarily large) together with 
a finite dimensional complex vector bundle $V \to M$  and 
a continuous map $f : M \to BG$ so that 
\[
    f_* ( [V] \cap [M]_K) = h \, . 
\]
Here we regard again $V \to M$ as an element in $K^0(M)$ and use the cap product pairing 
\[
   \cap : K^0(M) \times K_0(M) \to K_0(M) \, . 
\]
As $G$ is finitely presented we can assume that $f$ induces an isomorphism of fundamental groups. 
In view of Theorem \ref{Kess} we need to show that the class $[V] \cap [M]_K \in K_0(M)$ is 
of infinite $K$-area. 

Let $L \to M$ be the complex line bundle classified by $f^*(c)$. 
We pick a Hermitian connection on $L$ and denote by 
$\eta \in \Omega^2(M; i\R)$ the associated curvature form. Because the universal cover of $BG$ is contractible, the 
pull back $\pi^*(L) \to \widetilde M$ of $L$ to the universal cover $\pi  : \widetilde M  \to M$ is trivial. We fix a trivialization and denote the $1$-form 
associated to the pull back connection by $\omega \in \Omega^1( \widetilde M ; i \R)$. The curvature form $\pi^*(\eta)$ is 
equal to $ d\omega$, since  $\U(1)$ is 
abelian. However, the connection $1$-form $\omega$ is in general not invariant under the action of the 
deck transformation group on $\widetilde M$, because in this 
case the curvature form $\eta$  would be exact and hence $L \to M$ would be the trivial line bundle. 

We will now ``flatten'' the bundle $L \to M$ by scaling its curvature by a constant $0 < t < 1$. 
Unfortunately, this cannot be done directly, because the 
first Chern class of $L$ would no longer be integral. 

The following construction  originating from \cite{HS(2008)} gives a solution to this problem by considering infinite dimensional 
bundles. At first we consider the Hilbert space bundle 
\[
    E = \widetilde M \times_{G} l^2(G) \to M 
\]
where $l^2(G)$ is the set of square summable complex valued functions on $G$ and $G$ acts on the 
left of $l^2(G)$ by the formula 
\[
   ( \gamma \psi ) (x) = \psi ( x  \gamma ) 
\]
and on the right of $\widetilde M$ by $(x,g) \mapsto g^{-1} x$. 
Let $0 < t < 1$. We consider the $G$-invariant connection $1$-form on $\widetilde M \times l^2(G)$ which on the subbundle 
\[
    \widetilde M \times \C \cdot 1_g \subset \widetilde M \times l^2(G) 
\]
concides with $(g^{-1})^*(t\omega)$. Here $1_g \in l^2(G)$ is the characteristic function of $g \in G$. 
Because this one form is $G$-invariant, 
we obtain an induced connection $\nabla^t$ on 
the Hilbert space bundle $E$ whose curvature form is norm bounded by $t \cdot \| \eta \|$. In other words, 
the Hilbert space 
bundle $E$ can be equipped with holonomy representations which are arbitrarily close to the identity
(at some fixed scale). It hence remains to show that $E$  detects the $K$-homology class $ [V] \cap[M]_K$. 

However, by Kuiper's theorem, any Hilbert space bundle is trivial. Therefore we will first 
reduce the structure group of $E$ in a canonical way. This will result in  
finitely generated Hilbert $A_t$-module bundles $E_t \to M$ with appropriate unital $C^*$-algebras 
$A_t$, where $t \in (0,1]$. 
The algebras $A_t$ will depend on $t$. 

We fix a base point $p \in M$ and choose a point $q \in \widetilde M$ above $p$. The fibre over $p$ is then identified with the Hilbert space 
$l^2(G)$. Now we define 
\[
   A_t \subset B(l^2(G)) 
\]
as the norm-linear closure of all maps $l^2(G) \to  l^2(G)$ arising from parallel transport with respect to $\nabla^t$ along piecewise 
smooth loops in $M$ based at $p$. We furthermore define a bundle $E_t \to M$ whose fibre over $x \in M$ is given by the norm-linear closure  
in $\Hom(E|_p, E|_x)$ of all Hilbert space isomorphisms $E|_p \to E|_x$ aring from parallel transport with respect to $\nabla^t$ along 
piecewise smooth curves connecting $p$ with $x$. In this way we obtain, for each $t \in (0,1]$,  a free Hilbert $A_t$-module 
bundle of rank $1$ where the $A_t$-module structure on each fibre is induced by precomposition with parallel transport along 
piecewise smooth loops based at $p$. 

Now, on the one hand, parallel transport with respect to $\nabla^t$ induces a holonomy representation on $E_t \to M$ which, for small enough $t$, is arbitrarily closed to the identity (at a fixed scale which is independent of $t$).  

On the other hand, each of the algebras $A_t$ carries a canonical trace 
\[
   \tau_t : A_t \to \C \, , ~\tau_t(\psi) = \langle \psi(1_e), 1_e \rangle
\]
where  $1_e \in l^2(G)$ is the characteristic function of the neutral element $e \in G$ and $\langle - , - \rangle$ is the 
inner product on $l^2(G)$. For details we refer to \cite[Lemma 2.2]{HS(2008)}.
Using the Chern-Weil calculus from \cite{Schick(2005)} we obtain 
\[
      \tau_t (\langle [E_t] , [V] \cap [M]_K  \rangle) = \langle \exp(tc) , \ch (h) \rangle \in \R[t] \, . 
\]
See also \cite{HS(2008)}. The last polynomial is nonzero by our assumption $\langle c, \ch(h) \rangle \neq 0$. 
In particular, for infinitely many $k \in \N$ we have 
\[
     \langle [E_{1/k}] , [V] \cap [M]_K  \rangle \neq 0 \in K_0(A_{1/k}) \otimes \Q \, . 
\]
This implies that $[V] \cap [M]_K$ is a class of infinite $K$-area and together with Theorem \ref{Kess} finishes the proof of Theorem \ref{strnovikov}. 


\section{Homological invariance of essentialness} \label{large_homology} 


Recall from Definition \ref{homess} that a closed oriented manifold $M^n$ is called {\em essential}, if the classifying map $\phi : M \to B \pi_1(M)$
satisfies 
\[
    \phi_*([M]_H) \neq 0 \in H_n(B \pi_1(M) ; \Q) \, . 
\]

Essential manifolds obey Gromov's systolic inequality:

\begin{thm}[\cite{Gromov(1983)}]ÊLet $M$ be an essential Riemannian manifold of dimension $n$. 
Then there is a noncontractible loop $\gamma : [0,1] \to M$ satisfying
\[
   \ell(\gamma) \leq C(n)  \cdot \vol(M)^{1/n} 
\]
where the constant $C(n)$ depends only on $n$. 
\end{thm}

We show the following implication. 

\begin{thm} \label{thm2} Let $M$ be an oriented manifold of even dimension $2n$. If the class $[M]_H \in H_{2n}(M;\Q)$ has 
infinite $K$-area, then $M$ is essential. 
\end{thm} 

\begin{proof} Let $E \to M$ be the Mishchenko-Fomenko bundle. 
The proof of Theorem \ref{thm2} is based on the  commutative diagram 
\[
    \xymatrix{
    K_0(M)  \otimes \Q     \ar[d]^{=} \ar[rr]^-{\langle [E] , - \rangle}  &                                   &    K_0(C^*_{max} \pi_1(M) ) \otimes \Q \ar[d]^{=}  \\
    K_0(M) \otimes \Q     \ar[r]^-{\phi_*} \ar[d]^{\ch}_{\cong}                           &  K_0(B \pi_1(M))  \otimes \Q   \ar[r]^{\mu} 
    \ar[d]^{\ch}_{\cong} &     K_0(C^*_{max} \pi_1(M) )  \otimes \Q     \\
    H_{ev}(M; \Q)             \ar[r]^-{\phi_*}                                  &  H_{ev}(B \pi_1(M), \Q)                  &                                                                       }
 \]
Indeed, by Theorem \ref{Kess} the image of $\ch^{-1}([M]_H)$ under the map in the first line is non-zero. 
\end{proof} 

This theorem  implies  

\begin{itemize}
  \item Closed manifolds of infinite $K$-area in the sense of Gromov are essential.
 \item  (\cite{HS(2006), HS(2007)}) Area-enlargeable manifolds are essential (use 
            Proposition \ref{enlinfinite}).   
\end{itemize}

The second implication can be obtained without referring to $K$-theoretic considerations. 
This is carried out in \cite{BrunnHan}, where
several largeness properties of Riemannian manifolds are investigated from a 
purely homological point of view. The best results can be obtained for 
enlargeable manifolds, for which we have the following homological invariance result. 

\begin{thm}[\cite{BrunnHan}] \label{BrunnH} Let $G$ be a finitely presented group. Then there is 
a rational vector subspace
\[
     H^{sm}_*(BG; \Q) \subset H_*(BG; \Q)  
\]
with the following property:
Let $M$ be a closed oriented manifold of dimension $n$. Then $M$ is 
enlargeable, if and only if under the classifying map $\phi: M \to B \pi_1(M)$ we have 
\[
   \phi_*([M]) \notin H_n^{sm}(B\pi_1(M) ;\Q) 
\]
\end{thm} 

This result indeed implies that enlargeable manifolds are essential, because $0 \in H_n(B\pi_1(M) ; \Q)$ 
is contained in every vector subspace of $H_n(B\pi_1(M) ; \Q )$. 

Theorem \ref{BrunnH} can be seen as a form of homological invariance of enlargeability.  
The proof is based on the following definition of enlargeable homology classes in simplicial complexes.

\begin{defn}[\cite{BrunnHan}] \label{enlhom}Ê Let $C$ be a connected simplicial complex
with finitely generated fundamental group. A homology class  $h \in H_n(C; \Q)$ 
is called {\rm enlargeable}, if the following holds: Let 
$S \subset C$ be a finite subcomplex carrying $h$ and inducing a surjection on $\pi_1$. Then, for
every $\epsilon > 0$, there is a cover $\overline{C} \to C$ and an $\epsilon$-Lipschitz map 
$\overline{S} \to S^n$ which is constant outside a compact subset of $\overline{S}$ and sends the transfer
 ${\rm tr}(h) \in H^{l\! f}_n(\overline S ; \Q)$ in the locally finite homology of $\overline S$ to 
 a nonzero class in the reduced homology $\widetilde H_n( S^n ; \Q)$. 
Here $\overline{S}$ is the preimage of $S$ under the covering map $\overline{C} \to C$. 
\end{defn}

It is shown in  \cite{BrunnHan} that the condition for $c$ described in this definition is 
independent of the finite subcomplex $S \subset C$ carrying $c$ and inducing a surjection on $\pi_1$. 
Using this property it is not difficult to prove the following fact, see \cite[Prop. 3.4.]{BrunnHan}.  

\begin{prop}  Let $f: C \to D$ be a continuous map inducing an isomorphism 
of (finitely generated) fundamental groups. Then a class $h \in H_*(C; \Q)$ is 
enlargeable, if and only if the class $f_*(h) \in H_*(D; \Q)$ is enlargeable. 
\end{prop} 

From this Theorem \ref{BrunnH} follows, if we define $H_n^{sm}(BG ; \Q)$ as the 
subset consisting of all homology classes which are not enlargeable.

Theorem \ref{BrunnH} transforms the problem of determining enlargeable manifolds to 
a problem in group homology: Given a finitely generated group $G$, determine $H_*^{sm}(BG; \Q)$, the ``small''
group homology of $G$. In light of Theorem \ref{BrunnH} and the fact that the 
fundamental classes of enlargeable manifolds 
are of infinite $K$-area (see Proposition \ref{enlinfinite})  it 
is desirable to decide whether $H^{sm}_*(BG; \Q)$ can be non-zero. 
 This is answered in the positive in \cite[Theorem 4.8]{BrunnHan} by use of the Higman $4$-group \cite{Hig}.  
Together with Theorem \ref{BrunnH} this 
implies that there are essential manifolds which are not enlargeable, see \cite[Theorem 1.5]{BrunnHan}.  

In contrast to these positive results we do 
not know, whether there are essential manifolds which 
are not area-enlargeable. These manifolds would exist, if the following question had
an affirmative answer.

\begin{quest} \label{final_problem}  Is there an essential manifold whose fundamental class 
in singular homology $[M]_H$ 
is of finite $K$-area?
\end{quest}

\section{Rosenberg index and the reduced group $C^*$-algebra}

Let $M^n$ be a closed spin manifold. The method of Section \ref{index} can be used 
equally well to construct an index obstruction to positive scalar curvature
\[
  \alpha(M) \in K_n(C^*_{red} \pi_1(M)) \, . 
\]

The reduced group $C^*$-algebra does not share the universal property of the maximal group $C^*$-algebra which we used in the proof of Theorem \ref{Kess}. 

Exploiting the connection of $C^*_{red} \pi_1(M)$ to coarse geometry \cite{HR} we can still prove

\begin{thm}[\cite{HKRS(2007)}] \label{coarse} Let $M^n$ be an enlargeable spin manifold. Then 
\[
   \alpha(M) \neq 0 \in K_n(C^*_{red} \pi_1(M)) \, . 
\]
\end{thm} 

We do not know whether the same conclusion holds for area-enlargeable spin manifolds. This 
would be implied by an affirmative answer to the following question.

\begin{quest} Does Theorem \ref{Kess} remain true for the reduced group $C^*$-algebra?
\end{quest}

\end{document}